\documentclass[11pt]{amsart}

\usepackage[margin=1in]{geometry}

\usepackage{amsmath,amssymb,amsthm,mathtools}

\usepackage{microtype}

\usepackage[colorlinks=true,linkcolor=blue,citecolor=blue,urlcolor=blue]{hyperref}

\usepackage[capitalise,nameinlink,noabbrev]{cleveref}

\newtheorem{theorem}{Theorem}[section]

\newtheorem{lemma}[theorem]{Lemma}

\newtheorem{proposition}[theorem]{Proposition}

\newtheorem{conjecture}[theorem]{Conjecture}

\theoremstyle{remark}

\newtheorem*{remark}{Remark}

\newcommand{\E}{\mathbb E}

\newcommand{\Prob}{\mathbb P}

\newcommand{\cP}{\mathcal P}

\DeclareMathOperator{\avdeg}{\overline d}

\title{Nearly spanning regular subgraphs}

\author{Varun Sivashankar}
\address{Department of Mathematics, Princeton University}
\email{varunsiva@princeton.edu}

\date{}

\begin{document}

\begin{abstract}

Alon and Mubayi asked whether, for every integer $k\ge1$ and every $\varepsilon>0$, there exists $r_0=r_0(k,\varepsilon)$ such that every $r$-regular graph on $n$ vertices with $r\ge r_0$ contains a $k$-regular subgraph covering at least $(1-\varepsilon)n$ vertices. Previously, the conjecture was known for $k\in\{1,2\}$ and for $k$ and $r$ both even. We answer this question affirmatively for general $k$ and $r$ with $\varepsilon=O_k(r^{-1/2})$. For $k=2$ and every odd $r\ge3$, we show that $\varepsilon=1/(r^2-3)$ suffices which is best possible.

\end{abstract}

\maketitle
\section{Introduction}

Alon and Mubayi asked whether, for every integer $k\ge1$ and every $\varepsilon>0$, there exists $r_0=r_0(k,\varepsilon)$ such that every $r$-regular graph on $n$ vertices with $r\ge r_0$ contains a $k$-regular subgraph covering at least $(1-\varepsilon)n$ vertices. The question appears as Conjecture~3.1 in Alon's 2003 survey~\cite{Alon2003}, where it is attributed to discussions with Mubayi. The subgraph need not be induced.

For $k=1$, Vizing's theorem~\cite{Vizing1964} partitions the edges into at most $r+1$ matchings, one of which covers at least $rn/(r+1)$ vertices. For $k=2$, Alon~\cite[Theorem~3.2]{Alon2003} found a $2$-regular subgraph omitting $O(n\log r/(r\log\log r))$ vertices. Choi, Kim, Kostochka, Park and West~\cite{CKKPW2019} proved that every cubic graph has a $2$-regular subgraph omitting at most a proportion $1/6$ of its vertices, and that this proportion is best possible. They also asked for the sharp bound for larger odd degrees. A \emph{$2$-factor} is a spanning $2$-regular subgraph. When $r$ and $k$ are both even and $k\le r$, Petersen's $2$-factorization theorem~\cite{Petersen1891} gives a spanning $k$-regular subgraph.

We answer the Alon--Mubayi question affirmatively for every $k$ and determine the sharp answer for $k=2$. To state our results, let $f_k(G)$ be the largest number of vertices in a $k$-regular subgraph of $G$, and define $\delta_k(r)=\sup_G\bigl(1-f_k(G)/|V(G)|\bigr)$, where the supremum is over all $r$-regular graphs $G$. Thus $\delta_k(r)$ is the largest proportion of vertices that may have to be left uncovered.

\begin{itemize}
\item \Cref{thm:main} gives $\delta_k(r)\le100Ck/\sqrt r$ for all integers $1\le k\le r$, where $C\ge1$ is the absolute constant in \cref{lem:cjmm}. In particular, $\delta_k(r)\to0$ as $r\to\infty$ for every fixed $k$.
\item \Cref{thm:k2} gives $\delta_2(r)=1/(r^2-3)$ for every odd $r\ge3$.
\end{itemize}

The lower bounds depend on the parities of $k$ and $r$. Let $1\le k\le r$. For odd $k$ and even $r$, the graph $K_{r+1}$ gives $\delta_k(r)\ge1/(r+1)$, since a $k$-regular graph with $k$ odd has an even number of vertices. This is sharp for $k=1$ by Vizing's theorem. For odd $r$, the construction of O and West~\cite{OW2010}, discussed in \cref{subsec:sharp-construction}, gives $\delta_k(r)\ge1/(r^2-3)$ when $k$ is even and $\delta_k(r)\ge(r-k)/(r+1)^2$ when $k$ is odd. The latter bound is sharp for $k=1$ by their matching bound~\cite[Theorem~2.1]{OW2011}. These lower bounds suggest that, for fixed $k$, the correct rates are $r^{-2}$ for even $k$ and odd $r$, and $r^{-1}$ for odd $k$; see Conjectures~\ref{conj:even} and~\ref{conj:odd}.

The existence of a nonempty regular subgraph, without a bound on its order, has a longer history. Pyber, R\"odl and Szemer\'edi~\cite{PRS1995} proved that average degree $C_k\log\Delta$ forces a $k$-regular subgraph in a graph of maximum degree $\Delta$. Janzer and Sudakov~\cite{JS2023} improved this to $C_k\log\log\Delta$. We use a theorem of Chakraborti, Janzer, Methuku and Montgomery~\cite[Theorem~1.13]{CJMM2026}: if maximum degree is at most a fixed multiple of average degree $d$, then a $d'$-regular subgraph exists for every $d'$ up to a fixed multiple of $d$.

The two proofs in this paper use rather different ideas. For \cref{thm:main}, we start from a largest $k$-regular subgraph and enlarge it by simultaneous alternating-path exchanges. A short random walk produces many alternating paths, and a sampling argument selects paths with pairwise disjoint interiors. The expected-visit calculation is related to the alternating walks of Goel, Kapralov and Khanna~\cite{GKK2013}; see also Dani and Hayes~\cite{DH2025}. The final sampling and degree-trimming steps are close in spirit to arguments of Alon, Jiang, Miller and Pritikin~\cite{AJMP2003} and Janzer and Sudakov~\cite{JS2023}.

A \emph{bridge} is an edge whose deletion increases the number of connected components. For \cref{thm:k2}, we follow the decomposition at bridges used by Choi et al.~\cite{CKKPW2019}. We extend their local estimate to larger odd degrees using Edmonds' $b$-factor polytope~\cite{Edmonds1965}, then count vertices in the resulting bridge tree.

All graphs are finite and simple unless stated otherwise. We write $e(F)$, $\Delta(F)$ and $\avdeg(F)=2e(F)/|V(F)|$ for the number of edges, maximum degree and average degree of a nonempty graph $F$. The empty graph is allowed as a $k$-regular subgraph.
\section{A general nearly-spanning theorem}\label{sec:augmentation}

\begin{theorem}\label{thm:main}
There is an absolute constant $C\ge1$ such that, for all integers $1\le k\le r$, every $r$-regular graph on $n$ vertices contains a $k$-regular subgraph omitting at most $100Ckn/\sqrt r$ vertices. Equivalently, $\delta_k(r)\le100Ck/\sqrt r$.
\end{theorem}

We prove \cref{thm:main} by augmenting a largest $k$-regular subgraph. Let $H$ be a $k$-regular subgraph of an $r$-regular graph $G$. Write $S=V(H)$ for its covered vertices, $U=V(G)\setminus S$ for the uncovered vertices, and $Q=G-E(H)$, retaining all vertices of $G$.

A path between two vertices of $U$ is \emph{alternating} if its internal vertices lie in $S$ and its edges alternate between $Q$ and $H$, beginning and ending in $Q$. For example,
\[
u\ \xleftrightarrow{\ Q\ }\ x\ \xleftrightarrow{\ H\ }\ y
\xleftrightarrow{\ Q\ }\ v.
\]
We modify the selected subgraph $H$, keeping $G$ fixed. In this example, delete $xy$ from $H$ and add $ux$ and $yv$. Vertex $x$ loses one selected edge and gains one, and the same is true of $y$, so both keep degree $k$. The new vertices $u,v$ each gain one edge. The same exchange works along any alternating path: every internal vertex loses one $H$-edge and gains one $Q$-edge, while each endpoint gains one edge.

We therefore construct a simple auxiliary graph $F$ on a subset of $U$. Each edge $uv$ of $F$ represents an alternating path from $u$ to $v$. We require at most one path for each endpoint pair and require the represented paths to have pairwise disjoint interiors. If $F$ contains a nonempty $k$-regular subgraph $J$, its edges select exactly $k$ paths ending at every vertex of $J$. Exchanging along all of them simultaneously gives every vertex of $J$ degree $k$, while every old vertex still has degree $k$. Thus $H$ can be enlarged. The following lemma produces an auxiliary graph dense enough for this argument.

\begin{lemma}\label[lemma]{lem:auxiliary}
Let $0<\varepsilon\le1$, let $G$ be an $r$-regular graph on $n$ vertices, let $H\subseteq G$ be $k$-regular, and put $U=V(G)\setminus V(H)$. If $r>2k\ge2$ and
\[
|U|\ge\varepsilon n,
\qquad
\varepsilon\sqrt r\ge100,
\]
then there is a nonempty simple auxiliary graph $F$ with $V(F)\subseteq U$ such that
\[
\avdeg(F)\ge\frac{\varepsilon\sqrt r}{100},
\qquad
\Delta(F)\le20\avdeg(F).
\]
Each edge $uv$ of $F$ represents an alternating path from $u$ to $v$, and these paths have pairwise disjoint interiors.
\end{lemma}

We use the following special case of a theorem of Chakraborti, Janzer, Methuku and Montgomery~\cite[Theorem~1.13]{CJMM2026}.

\begin{samepage}
\begin{lemma}[Chakraborti--Janzer--Methuku--Montgomery]\label[lemma]{lem:cjmm}
There is an absolute constant $C\ge1$ such that, for every integer $k\ge1$, every nonempty graph $F$ satisfying
\[
\avdeg(F)\ge Ck,
\qquad
\Delta(F)\le20\avdeg(F)
\]
contains a nonempty $k$-regular subgraph.
\end{lemma}
\end{samepage}

\begin{proof}[Proof of \cref{thm:main}]
Take $C$ to be the constant in \cref{lem:cjmm}. If $100Ck\ge\sqrt r$, the empty subgraph gives the stated bound. Otherwise, set $\varepsilon=100Ck/\sqrt r$. Then $0<\varepsilon<1$, $\varepsilon\sqrt r=100Ck\ge100$, and $r>(100Ck)^2>2k$. Choose a $k$-regular subgraph $H$ covering as many vertices as possible, and put $U=V(G)\setminus V(H)$.

Suppose that $|U|\ge\varepsilon n$. By \cref{lem:auxiliary}, there is an auxiliary graph $F$ with
\[
\avdeg(F)\ge\frac{\varepsilon\sqrt r}{100}=Ck,
\qquad
\Delta(F)\le20\avdeg(F).
\]
Lemma~\ref{lem:cjmm} gives a nonempty $k$-regular subgraph $J$ of $F$. The simultaneous exchanges described above then enlarge $H$, contradicting its choice. Hence $|U|<\varepsilon n=100Ckn/\sqrt r$, as required.
\end{proof}

\section{Proof of Lemma~\ref{lem:auxiliary}}\label{sec:paths}

Fix $G,H,\varepsilon$ as in \cref{lem:auxiliary}. Write $S=V(H)$, $U=V(G)\setminus S$, and $Q=G-E(H)$, retaining all vertices of $G$. Set $N=|U|\ge\varepsilon n$ and $q=r-k$. Since $r>2k$, we have $q>r/2$. The degrees in $Q$ are
\[
d_x=\deg_Q(x)=
\begin{cases}
r,&x\in U,\\
q,&x\in S.
\end{cases}
\]
Thus $d_x$ is simply the degree of $x$ in the graph of unused edges.

\subsection{Expected length and visits}\label{subsec:walk-estimates}

Starting from $u\in U$, take a uniformly chosen incident $Q$-edge. Stop if the arrival vertex lies in $U$; otherwise take a uniformly chosen incident $H$-edge and repeat. These steps are always possible: $d_x>0$ everywhere, and every vertex of $S$ has $k\ge1$ incident $H$-edges. The graphs stay fixed, and the walk may revisit vertices and edges.

A \emph{round} consists of a $Q$-step and, if the walk has not stopped, the following $H$-step. Let $\tau$ be the number of rounds, possibly infinite, and write $\Prob_u,\E_u$ for probability and expectation with start $u$.

For each vertex $x$ and integer $j\ge1$, define
\[
b_j(x)=\sum_{u\in U}
\Prob_u(\text{round $j$ occurs and starts at $x$}).
\]
This counts expected $Q$-departures from $x$ in round $j$ alone, summed over all starting vertices. For $m\ge0$, set
\[
a_m(x)=\sum_{j=1}^{m}b_j(x).
\]
Thus $a_m(x)$ counts expected departures in the first $m$ rounds, including repeated visits. In particular, $a_0(x)=0$ by the empty-sum convention. We claim that
\begin{equation}\label{eq:visits}
a_m(x)\le d_x/r\qquad\text{for every $m\ge0$ and every $x$.}
\end{equation}

We prove the claim by induction. The base case follows from $a_0(x)=0$. Suppose the claim holds for $m$. If $x\in U$, only the walk starting at $x$ can depart from $x$, and only in its first round. Hence $a_{m+1}(x)=1=d_x/r$.

Fix $s\in S$. A departure from $s$ in round $j+1$ follows an $H$-step $ys$ in round $j$, preceded by a $Q$-step $xy$, where $y\in N_H(s)$ and $x\in N_Q(y)$. Here $N_H(s)$ and $N_Q(y)$ denote the respective neighbour sets. Conversely, taking these two steps leads to a departure from $s$ in the next round.

Given a departure from $x$, the walk chooses $xy$ with probability $1/d_x$, then $ys$ with probability $1/k$. The combined probability is $1/(k d_x)$. For a walk starting at $u$, the probability of this sequence in round $j$ is therefore
\[
\Prob_u(\text{round $j$ occurs and starts at $x$})\,\frac1{k d_x}.
\]
This includes the probability that the departure from $x$ occurs. Summing over starting vertices and possible $x,y$ gives
\[
b_{j+1}(s)=\sum_{y\in N_H(s)}\sum_{x\in N_Q(y)}\frac{b_j(x)}{k d_x}.
\]

No walk starts at $s$, so $b_1(s)=0$. Its departures in the first $m+1$ rounds are therefore exactly those in rounds $2,\ldots,m+1$. Summing the preceding identity over $j=1,\ldots,m$ gives
\[
\begin{aligned}
a_{m+1}(s)
&=\sum_{j=1}^{m}b_{j+1}(s)\\
&=\sum_{y\in N_H(s)}\sum_{x\in N_Q(y)}\frac{\sum_{j=1}^{m}b_j(x)}{k d_x}\\
&=\sum_{y\in N_H(s)}\sum_{x\in N_Q(y)}\frac{a_m(x)}{k d_x}.
\end{aligned}
\]
Repeated visits contribute separately; these equalities use linearity of expectation and do not require independent visits. By the induction hypothesis, each summand is at most $1/(kr)$. There are $k$ choices of $y$ and $q$ choices of $x$ for each $y$, so
\[
a_{m+1}(s)\le\frac{kq}{kr}=\frac qr=\frac{d_s}{r}.
\]
This proves~\eqref{eq:visits}.

Let $a(x)=\lim_m a_m(x)$. Since each round has one $Q$-departure, summing nonnegative terms gives
\begin{equation}\label{eq:length}
\sum_{u\in U}\E_u\tau=\sum_x a(x)\le\sum_x d_x/r\le n.
\end{equation}
In particular, for each starting vertex, the probability that the walk continues forever is zero.

We next count visits to a fixed $s\in S$. A \emph{$Q$-arrival} at $s$ occurs when the walk reaches $s$ along a $Q$-edge; an \emph{$H$-arrival} occurs when it reaches $s$ along an $H$-edge. These occur after the first and second steps of a round, respectively. We count repeated visits separately and sum expectations over all starting vertices.

\emph{$Q$-arrivals.} The expected number of $Q$-arrivals at $s$ is the expected number of $Q$-departures from vertices $x\in N_Q(s)$ that reach $s$. The expected number of departures from $x$ is $a(x)$, and each departure reaches $s$ with probability $1/d_x$. Thus the expected number of $Q$-arrivals at $s$ equals
\[
\sum_{x\in N_Q(s)}\frac{a(x)}{d_x}
\le\sum_{x\in N_Q(s)}\frac1r
=\frac qr.
\]
The inequality uses $a(x)\le d_x/r$, and the last equality uses $|N_Q(s)|=q$.

\emph{$H$-arrivals.} Fix $s\in S$. Every $H$-arrival at $s$ in round $j$ is followed by a $Q$-departure from $s$ in round $j+1$, since the walk does not stop in $S$. Conversely, every $Q$-departure from $s$ follows such an $H$-arrival, since no walk starts at $s$. Hence the expected number of $H$-arrivals in round $j$ is $b_{j+1}(s)$. Summing over rounds, and using $b_1(s)=0$, gives
\[
\sum_{j\ge1}b_{j+1}(s)=\sum_{j\ge1}b_j(s)=a(s)\le\frac qr
\qquad(s\in S).
\]
Every visit to $s\in S$ is either a $Q$-arrival or an $H$-arrival. We have therefore proved the following:
\begin{enumerate}
\item The expected number of rounds, summed over all starting vertices $u\in U$, is at most $n$, by~\eqref{eq:length}.
\item For each fixed $s\in S$, the expected number of visits to $s$, summed over all starting vertices and counting repeated visits, is at most $2q/r\le2$.
\end{enumerate}

\subsection{Short simple paths}

Recall that $N=|U|$ is the number of vertices left uncovered by $H$, and $N\ge\varepsilon n$.

Set $L=\lfloor3/\varepsilon\rfloor$. Call a walk \emph{good} if it stops within $L$ rounds and has no repeated vertex. Since $\tau$ is integer-valued, \eqref{eq:length} and Markov's inequality give
\[
\sum_{u\in U}\Prob_u(\tau>L)\le\frac1{L+1}\sum_{u\in U}\E_u\tau\le\frac n{L+1}\le\frac{\varepsilon n}{3}\le\frac N3.
\]
We next bound the probability of a repetition in the first $L$ rounds.

Condition on the history before a round, with current vertex $x$. The probability of a $Q$-arrival at any specified vertex is at most $1/d_x\le1/q$. For a specified $s\in S$, the probability of an $H$-arrival at $s$ is
\[
\frac{|N_Q(x)\cap N_H(s)|}{k d_x}\le\frac1q,
\]
since $|N_H(s)|=k$. There is no $H$-arrival when the $Q$-step ends the walk.

Suppose round $j$ begins, where $j\le L$. Each of the preceding $j-1$ rounds had both a $Q$-step and an $H$-step, since the walk has not stopped. These gave two arrivals in $S$ per round. Thus at most $2(j-1)\le2L$ distinct vertices of $S$ have already been visited.

Given the history before round $j$, each of its at most two arrivals hits any specified previously visited vertex with probability at most $1/q$. A union bound over the two arrivals and at most $2L$ previously visited vertices of $S$ gives
\[
2\cdot(2L)\cdot\frac1q=\frac{4L}{q}
\]
as an upper bound on the probability of hitting a previously visited vertex of $S$ in that round. Summing over at most $L$ rounds gives a bound of $4L^2/q$.

The two arrivals within a single round cannot coincide, since they are joined by an $H$-edge and there are no loops.

The remaining possible repetition is that the walk ends at its starting vertex $u$. The walk stops on its first return to $U$, so all vertices between its start and end lie in $S$. Only a $Q$-step can return to $u$, with probability at most $1/q$ per round, giving a total bound of $L/q$. These cases cover every possible repetition. Since $q>r/2$, $L\le3/\varepsilon$ and $\varepsilon\sqrt r\ge100$,
\[
\Prob_u(\tau\le L\text{ but the walk is not good})
\le\frac{4L^2+L}{q}\le\frac{10L^2}{r}
\le\frac{90}{\varepsilon^2r}<\frac1{15}.
\]
Together with the bound on long walks, this gives
\begin{equation}\label{eq:good}
\sum_{u\in U}\Prob_u(\text{good})\ge N-\frac N3-\frac N{15}=\frac{3N}{5}.
\end{equation}
Every good walk is a simple alternating path with distinct endpoints in $U$ and at most $2L$ internal vertices, all in $S$.

\subsection{Path weights}

Let $\cP$ be the family of all unoriented good paths. We assign weights using their traversal probabilities. If $P$ has endpoints $u,v$ and $t$ $Q$-edges, the probability of traversing exactly $P$ from $u$ to $v$ is
\[
\pi(P)=\frac1{r q^{t-1}k^{t-1}}.
\]
The same formula holds from $v$ to $u$. Define
\[
w(P)=\frac r{2L}\pi(P).
\]
Each unoriented path is counted from both endpoints in~\eqref{eq:good}, so
\[
W:=\sum_{P\in\cP}w(P)
=\frac r{4L}\sum_{u\in U}\Prob_u(\text{good})
\ge\frac{3rN}{20L}.
\]

We need three elementary bounds on these weights. For distinct $u,v\in U$ and $s\in S$,
\begin{equation}\label{eq:weights}
\begin{array}{l|c}
\text{Paths being counted}&\text{Total weight at most}\\ \hline
\text{with endpoints }u,v&1\\
\text{with endpoint }u\text{ and internal vertex }s&2\\
\text{with internal vertex }s&r/(2L).
\end{array}
\end{equation}
For the first two rows, a walk starting at $u$ ends at $v$ within $L$ rounds with probability at most $L/q$, and visits $s$ within those rounds with probability at most $2L/q$. Multiplying by $r/(2L)$ and using $q>r/2$ gives the stated bounds.
For the last row, the expected number of visits to a fixed $s\in S$, summed over all starting vertices, is at most $2$ by \cref{subsec:walk-estimates}. A good path visits $s$ at most once, and each unoriented path is counted from its two endpoints. Hence
\[
\sum_{P:\,s\text{ internal}}w(P)
\le \frac r{4L}\cdot2=\frac r{2L}.
\]

\subsection{Selecting paths with disjoint interiors}

We now select paths at random and delete enough of them to make their interiors disjoint. Set
\[
p=\frac1{4\sqrt r},
\qquad
D=\frac{pr}{2L}=\frac{\sqrt r}{8L}.
\]
Since $L/\sqrt r\le3/100$, we have $D\ge4$. Form $X$ by including each vertex of $U$ independently with probability $p$.
For each unordered pair $\{u,v\}\subseteq U$, independently choose one path $P$ joining it with probability $w(P)$, or choose none with the remaining probability. The first row of~\eqref{eq:weights} makes this possible. Make these choices independently of $X$, and keep a chosen path only if both endpoints lie in $X$.

The chosen paths represent a simple graph $A$ on $X$. Each path $P$ is chosen with probability $p^2w(P)$, so
\begin{equation}\label{eq:edges}
M:=\E e(A)=Wp^2,
\qquad
\E|X|=Np.
\end{equation}

Let $T$ count unordered pairs of chosen paths whose interiors meet. Paths with the same endpoint pair never coexist. We distinguish whether the two paths have a common endpoint.

If two paths have one common endpoint, then for a fixed path $P$ the total weight of paths sharing both an endpoint and an internal vertex with $P$ is at most
\[
2\cdot(2L)\cdot2=8L.
\]
Here we choose one of the two endpoints of $P$, one of at most $2L$ internal vertices, and use the second row of~\eqref{eq:weights}. All three endpoints must lie in $X$. Summing over $P$ counts each unordered pair twice, so these pairs contribute at most $4LWp^3$ to $\E T$.

If the paths have no common endpoint, then for a fixed $P$ the total weight of paths whose interior meets that of $P$ is at most
\[
(2L)\frac r{2L}=r
\]
by the last row of~\eqref{eq:weights}. All four endpoints must lie in $X$. Again dividing by two, these pairs contribute at most $(r/2)Wp^4$. Consequently,
\begin{equation}\label{eq:conflicts}
\E T\le4LWp^3+\frac r2Wp^4
=M\left(\frac L{\sqrt r}+\frac1{32}\right)<\frac M{12},
\end{equation}
using $L/\sqrt r\le3/100$.

We also control the auxiliary degrees. Conditional on $u\in X$, only the other endpoint $v$ still needs to lie in $X$. Thus each other vertex $v\in U$ contributes an edge at $u$ with probability $p$ times the total weight of paths with endpoints $u,v$. These indicators are independent, since they depend on different endpoint pairs and different membership choices. By the definition of $w(P)$,
\[
\begin{aligned}
\E[\deg_A(u)\mid u\in X]
&=p\sum_{P:\,u\text{ endpoint}}w(P)
=\frac{pr}{2L}\sum_{P:\,u\text{ endpoint}}\pi(P)\\
&=\frac{pr}{2L}\Prob_u(\text{good})\le D.
\end{aligned}
\]
Here the probabilities $\pi(P)$ sum to $\Prob_u(\text{good})$, since every good walk starting at $u$ follows exactly one of these paths.
The variance of a sum of independent zero-one variables is at most its mean, and therefore
\[
\E[\deg_A(u)^2\mid u\in X]
\le(D+1)\E[\deg_A(u)\mid u\in X].
\]

Delete all auxiliary edges incident to vertices whose degree in $A$ exceeds $6D$. Let $Z$ be the number of edges deleted. Every deleted edge has at least one such endpoint, so summing their degrees counts every deleted edge at least once. Since $z\le z^2/(6D)$ for $z>6D$, we have
\[
\begin{aligned}
Z
&\le\sum_{\substack{u\in X\\\deg_A(u)>6D}}\deg_A(u)\\
&\le\frac1{6D}\sum_{\substack{u\in X\\\deg_A(u)>6D}}\deg_A(u)^2
\le\frac1{6D}\sum_{u\in X}\deg_A(u)^2.
\end{aligned}
\]
Taking expectations, using $\Prob(u\in X)=p$ and $D \geq 4$ gives
\[
\begin{aligned}
\E Z
&\le\frac1{6D}\E\left[\sum_{u\in X}\deg_A(u)^2\right]\\
&=\frac1{6D}\sum_{u\in U}p\,\E[\deg_A(u)^2\mid u\in X]\\
&\le\frac{D+1}{6D}\sum_{u\in U}p\,\E[\deg_A(u)\mid u\in X]
&= \frac{D+1}{6D} 2M
&\leq\frac{5M}{12}.
\end{aligned}
\]

Finally process the intersecting pairs in any fixed order. Whenever both paths are still present, delete one of them. At most $T$ additional auxiliary edges are deleted. Let $F$ be the remaining auxiliary graph on $X$, keeping isolates. The represented paths now have pairwise disjoint interiors, $\Delta(F)\le6D$, and
\[
\E e(F)\ge M-\E Z-\E T>\frac M2\ge\frac{3NpD}{20},
\]
where $M=Wp^2\ge(3/10)NpD$. Comparing the number of edges with the number of sampled vertices gives
\[
\E\left[e(F)-\frac{3D}{20}|X|\right]
>\frac{3NpD}{20}-\frac{3D}{20}Np=0.
\]
Hence some outcome satisfies $e(F)>(3D/20)|X|$. For that outcome,
\[
\avdeg(F)>\frac{3D}{10}=\frac{3\sqrt r}{80L}
\ge\frac{\varepsilon\sqrt r}{100},
\qquad
\Delta(F)\le6D<20\avdeg(F),
\]
where we used $L\le3/\varepsilon$. This proves \cref{lem:auxiliary}.

\section{The sharp degree-two case}\label{sec:k2}

\begin{theorem}\label{thm:k2}
For every odd integer $r\ge3$,
\[
\delta_2(r)=\frac1{r^2-3}.
\]
\end{theorem}

For even $r$, Petersen's $2$-factorization theorem~\cite{Petersen1891} gives $\delta_2(r)=0$. The case $r=3$ of \cref{thm:k2} is due to Choi et al.~\cite{CKKPW2019}.

We follow the approach of Choi et al.: delete the bridges and bound the number of vertices omitted in each remaining component. Their argument does not generalize to odd $r > 3$. We extend their approach by adding loops to represent omitted vertices and applying the degree-two case of Edmonds' theorem. The resulting estimate also proves the bound proposed by Choi et al. in terms of bridges and degree deficiency; see \cref{prop:k2-bridges}.

\subsection{The local estimate}

Throughout this section, let $r\ge3$ be odd. Multigraphs may have loops, each contributing two to the degree. For a multigraph $B$ with $\Delta(B)\le r$, define its degree deficiency by
\[
D_r(B)=\sum_{v\in V(B)}(r-d_B(v)).
\]
For $r=3$, the following estimate appears in the proof of Choi et al.'s Theorem~2.3~\cite{CKKPW2019}.

\begin{lemma}\label[lemma]{lem:k2-local}
Let $r\ge3$ be odd, and let $B$ be a connected bridgeless multigraph with $\Delta(B)\le r$. Then $B$ has a $2$-regular subgraph omitting at most
\[
\left\lfloor\frac{D_r(B)}r\right\rfloor
\]
vertices.
\end{lemma}

To prove the lemma, we add one loop at each vertex of $B$. A $2$-factor of the enlarged graph consists of a $2$-regular subgraph of $B$ together with an added loop at each omitted vertex. Thus we seek a $2$-factor using few added loops.

We use the following consequence of Edmonds' theorem~\cite{Edmonds1965}; see Kobayashi~\cite[Section~2]{Kobayashi2022} for a formulation allowing loops. For a multigraph $J$ and $S\subseteq V(J)$, let $\partial_J S$ denote the set of edges with exactly one endpoint in $S$.

\begin{lemma}\label[lemma]{lem:edmonds-two-factor}
Let $J$ be a multigraph and assign a number $x_e\in[0,1]$ to each edge $e\in E(J)$. Suppose that the weights incident with each vertex sum to two, counting loops twice, and that
\begin{equation}\label{eq:two-factor-cut}
\sum_{e\in\partial_J S\setminus A}x_e
+\sum_{e\in A}(1-x_e)\ge1
\end{equation}
for every $S\subseteq V(J)$ and every subset $A\subseteq\partial_J S$ of odd size. Then there is a probability distribution on the $2$-factors $F$ of $J$ such that $\Prob(e\in E(F))=x_e$ for every edge $e$.
\end{lemma}

The cut condition has a simple interpretation. A $2$-factor uses an even number of edges across every cut, so its selected cut edges cannot equal an odd set $A$. It must therefore use an edge of $\partial_J S\setminus A$ or omit an edge of $A$. Taking expectations gives \eqref{eq:two-factor-cut}. Edmonds' theorem states that these conditions, together with the degree and weight conditions, suffice.

\begin{proof}[Proof of \cref{lem:k2-local}]
Add a loop $\ell_v$ at every vertex $v$ of $B$, obtaining a multigraph $B^+$. Set
\[
x_e=\frac2r\quad(e\in E(B)),
\qquad
x_{\ell_v}=\frac{r-d_B(v)}r\quad(v\in V(B)).
\]
These weights lie in $[0,1]$. At each vertex $v$, their sum, counting loops twice, is
\[
\frac{2d_B(v)}r+2\frac{r-d_B(v)}r=2.
\]

We check \eqref{eq:two-factor-cut}. Fix $S\subseteq V(B^+)$ and an odd set $A\subseteq\partial_{B^+}S$, and write $c=|\partial_{B^+}S|$ and $a=|A|$. Loops do not cross cuts, so every edge in this cut belongs to $B$ and has weight $2/r$. Since $a\ge1$ and $B$ has no bridge, $c\ge2$. If $a=1$, the left side of \eqref{eq:two-factor-cut} is
\[
\frac2r(c-1)+1-\frac2r
=1+\frac{2(c-2)}r\ge1.
\]
If $a\ge3$, then $c\ge a$ and $r\ge3$ give
\[
\frac2r(c-a)+\left(1-\frac2r\right)a
\ge\frac a3\ge1.
\]

Lemma~\ref{lem:edmonds-two-factor} therefore gives a random $2$-factor $F$ of $B^+$ for which the expected number of added loops is
\[
\E\bigl|\{v:\ell_v\in E(F)\}\bigr|
=\sum_{v\in V(B)}x_{\ell_v}
=\frac{D_r(B)}r.
\]
Some $2$-factor consequently uses at most $\lfloor D_r(B)/r\rfloor$ added loops. Delete the vertices carrying those loops. No other selected edge is incident with such a vertex, so every remaining vertex keeps degree two. The remaining subgraph lies in $B$ and has the required order.
\end{proof}

\subsection{Bridges and the upper bound}

Recall that $f_2(G)$ is the largest number of vertices in a $2$-regular subgraph of $G$. The local estimate yields the following bound, proposed with a ceiling in place of the floor by Choi et al.~\cite[Section~1]{CKKPW2019}.

\begin{proposition}\label[proposition]{prop:k2-bridges}
Let $r\ge3$ be odd, and let $G$ be a multigraph with $\Delta(G)\le r$, $c$ bridges and degree deficiency $d=D_r(G)=\sum_{v\in V(G)}(r-d_G(v))$. Then
\[
|V(G)|-f_2(G)\le
\max\left\{0,\left\lfloor\frac{d+c-1}{r-1}\right\rfloor\right\}.
\]
\end{proposition}

\begin{proof}
Delete all bridges of $G$, obtaining components $B_1,\ldots,B_m$. Contracting each $B_i$ to a vertex $i$ gives a forest $T$ with $c$ edges. The degree deficiency of $B_i$ is
\[
D_r(B_i)=\sum_{v\in V(B_i)}(r-d_G(v))+d_T(i),
\]
since each bridge incident with $B_i$ contributes one additional unit of deficiency after deletion.

Put $t_i=\lfloor D_r(B_i)/r\rfloor$ and $t=\sum_i t_i$. By \cref{lem:k2-local}, the union of suitable $2$-regular subgraphs of the $B_i$ omits at most $t$ vertices. If $t=0$, the assertion follows. Otherwise, let $I=\{i:t_i>0\}$, so $1\le|I|\le t$.

In the sum $\sum_{i\in I}d_T(i)$, each edge of $T$ is counted at most once except for edges of the induced forest $T[I]$, which are counted twice. Hence
\[
\sum_{i\in I}d_T(i)
\le c+e(T[I])\le c+|I|-1 \le c+t-1.
\]
Since $rt_i\le D_r(B_i)$ and each term in the sum defining $d$ is nonnegative, we obtain
\[
\begin{aligned}
rt
&\le\sum_{i\in I}D_r(B_i)
\le d+\sum_{i\in I}d_T(i)
\le d+c+t-1.
\end{aligned}
\]
Thus $t\le(d+c-1)/(r-1)$, and taking the floor proves the assertion.
\end{proof}

In particular, an $r$-regular graph with at most $r-1$ bridges has a $2$-factor, recovering the theorem of Hanson, Loten and Toft~\cite{HLT1998}. For connected graphs, van den Heuvel and Toft~\cite[Theorem~1.9]{HT2026} replace the bridge bound by the weaker requirement that at most \(r-1\) components obtained by deleting the bridges are incident with exactly one bridge.

\begin{proof}[Proof of the upper bound in \cref{thm:k2}]
Let $G$ be a connected simple $r$-regular graph on $n$ vertices, where $r\ge3$ is odd, and let $c$ be its number of bridges. If $c=0$, then \cref{lem:k2-local} gives a $2$-factor. If $c\ge1$, the bridge-count bound of O and West~\cite[Theorem~3.4]{OW2010} gives
\[
c-1\le\frac{(r-1)(n-2(r+2))}{r^2-3}.
\]
Since $D_r(G)=0$, \cref{prop:k2-bridges} now yields
\[
|V(G)|-f_2(G)
\le\left\lfloor\frac{c-1}{r-1}\right\rfloor
\le\left\lfloor\frac{n-2(r+2)}{r^2-3}\right\rfloor
<\frac{n}{r^2-3}.
\]
Applying this bound separately to the connected components proves $\delta_2(r)\le1/(r^2-3)$.
\end{proof}

\subsection{The sharp construction}\label{subsec:sharp-construction}

Sharpness follows from the construction of O and West~\cite[Construction~2.1]{OW2010}, which we describe in our notation. Define a graph $L_r$ on $r+2$ vertices as follows. Start with $K_{r+2}$, choose three distinct vertices $v,a,b$, delete the two edges $va,vb$, and delete a perfect matching on the remaining $r-1$ vertices. Then
$d_{L_r}(v)=r-1$ and $d_{L_r}(x)=r$ for $x\ne v$.

The graph $L_r$ has a $2$-factor. Indeed, $L_r-v$ is $K_{r+1}$ with $(r-1)/2$ edges deleted. Since $K_{r+1}$ decomposes into $r$ perfect matchings, one of them avoids all the deleted edges. Let $M$ be such a matching. Removing $M$ from $L_r$ leaves an $(r-1)$-regular graph on all $r+2$ vertices; since $r-1$ is even, Petersen's theorem decomposes it into $2$-factors.

For each integer $t$, take a tree with exactly $t$ nonleaf vertices, each of degree $r$. Such a tree has $\ell=(r-2)t+2$ leaves. Delete every leaf and, in its place, attach a fresh copy of $L_r$ by joining the former neighbor of the leaf to the vertex corresponding to $v$ in that copy. The resulting graph $G_t$ is simple and $r$-regular, with
\[
|V(G_t)|=t+(r+2)((r-2)t+2)=(r^2-3)t+2(r+2).
\]
Every edge of $G_t$ outside the copies of $L_r$ is a bridge, since contracting each copy of $L_r$ to a single vertex recovers the original tree. A $2$-regular subgraph is a union of cycles and hence uses no bridge, so none of the $t$ tree vertices can be covered. Conversely, each copy of $L_r$ has a spanning $2$-factor, so all other vertices can be covered. Therefore
\[
1-\frac{f_2(G_t)}{|V(G_t)|}
=\frac{t}{(r^2-3)t+2(r+2)}
\longrightarrow\frac1{r^2-3}.
\]
This proves the reverse inequality and completes the proof of \cref{thm:k2}.

\begin{remark}
The same construction gives lower bounds for other values of $k$. First suppose that $k$ is even. If a $k$-regular subgraph $H$ contained a bridge of $G_t$, the degree sum in $H$ over the vertices on one side of the bridge would be twice the number of edges of $H$ on that side, plus one for the bridge. This is odd, whereas every degree in $H$ is even. Thus $H$ uses no bridge and omits all $t$ tree vertices. Letting $t\to\infty$ gives $\delta_k(r)\ge1/(r^2-3)$.

Now let $1\le k\le r$ be odd and take $t=1$. The graph $G_1$ consists of a vertex $z$ joined by one bridge to each of $r$ disjoint copies of $L_r$, and has $1+r(r+2)=(r+1)^2$ vertices.

Let $H$ be a $k$-regular subgraph of $G_1$. If $z\in V(H)$, exactly $k$ of the $r$ bridges incident with $z$ belong to $H$; if $z\notin V(H)$, none does. Hence at least $r-k$ copies of $L_r$ have their joining bridge absent from $H$.

Fix such a copy $B$. Its joining bridge is the only edge of $G_1$ between $B$ and the rest of the graph. If every vertex of $B$ were covered by $H$, all $k$ selected edges at each vertex of $B$ would therefore lie inside $B$. The degree sum of the restriction of $H$ to $B$ would be $k(r+2)$, which is odd, contradicting the handshaking lemma. So each of these $r-k$ disjoint copies has at least one uncovered vertex, and
\[
\delta_k(r)\ge1-\frac{f_k(G_1)}{|V(G_1)|}
\ge\frac{r-k}{(r+1)^2}.
\]
For every fixed odd $k$, this is $\Omega_k(r^{-1})$ as $r\to\infty$ through odd integers.
\end{remark}

\section{Concluding remarks}

The sharp results for $k=1$ and $k=2$ suggest the following conjectures.

\begin{conjecture}\label{conj:even}
For every fixed even integer $k\ge4$, we have $\delta_k(r)=\Theta_k(r^{-2})$ as $r\to\infty$ through odd integers.
\end{conjecture}

\begin{conjecture}\label{conj:odd}
For every fixed odd integer $k\ge3$, we have $\delta_k(r)=\Theta_k(r^{-1})$ as $r\to\infty$.
\end{conjecture}

Our general upper bound is $O(k/\sqrt r)$. Allowing the alternating paths to share internal vertices while remaining edge-disjoint improves this to $O(\sqrt{k/r})$. For fixed even $k$, a more involved flow argument gives $O_k(\log r/r)$. We omit these longer arguments, as neither reaches the conjectured rate. It would also be interesting to determine the optimal dependence on both $k$ and $r$ when $k$ grows with $r$.

\section*{Acknowledgements}

The proof was discovered using ChatGPT Astra (Ultra). We thank Noga Alon for helpful discussions.

\end{document}